\documentclass{amsart}

\newtheorem{theorem}{Theorem}[section]
\newtheorem{lemma}[theorem]{Lemma}

\theoremstyle{definition}
\newtheorem{definition}[theorem]{Definition}
\newtheorem{example}[theorem]{Example}

\newtheorem{proposition}[theorem]{Proposition}

\theoremstyle{remark}
\newtheorem{remark}[theorem]{Remark}

\numberwithin{equation}{section}

\usepackage{stmaryrd}
\usepackage{mathbbol}
\usepackage{amsfonts}
\usepackage{mathrsfs}
\usepackage{bbm}
\usepackage{lineno,hyperref}
\usepackage{latexsym}
\usepackage{enumerate}

\def\D{\mathcal D}
\def\K{\mathcal K}
\def\A{\mathscr A}
\def\Q{\mathcal Q}
\def\a{\alpha}
\def\b{\beta}
\def\l{\lambda}
\def\R{\mathbb R}

\def\E{\mathbb E}
\def\F{\mathscr F}
\def\P{\mathbb P}

\def\o{\omega}
\def\O{\Omega}

\def\t{\tau}
\def\e{\varepsilon}

\def\p{\psi}
\def\r{\rightarrow}

\def\s{\sigma}

\def\g{\gamma}

\def\i{\infty}

\def\oo{\mathcal O}
\def\B{\mathcal B}

\begin{document}

\title[Weak Pullback Mean Random Attractors for SEE and Applications]
{Weak Pullback  Mean Random Attractors for Stochastic Evolution Equations
and Applications}

\author{Anhui Gu}
\address{School of Mathematics and Statistics, Southwest
University, Chongqing 400715, China}

\email{gahui@swu.edu.cn}

\subjclass[2000]{Primary 35B40; Secondary 35B41; 37L30}



\keywords{Weak pullback mean random attractor, general diffusion term, stochastic evolution equation.}

\begin{abstract}
In this paper, we investigate the existence and uniqueness of weak pullback mean random attractors
for abstract
stochastic evolution equations with general diffusion terms in Bochner spaces.
As applications, the existence and uniqueness of weak pullback mean random attractors for some
stochastic models such as stochastic reaction-diffusion equations, the
stochastic $p$-Laplace equation and stochastic porous media equations are established.
\end{abstract}

\maketitle



\section{Introduction}\label{sect0}

In this work, we aim to consider the long-term behavior for the
following non-autonomous stochastic abstract evolution
equation on an open bounded domain $\oo\subseteq \R^n$:
\begin{equation} \label{sys00}
du=\big(A(t, u)+F(u)+g(t, x)\big)dt+\e G(t, u)dW(t),
\quad x\in \oo, t>\t,
\end{equation}
with initial-boundary condition,
where $\t\in \R$, $F$ is a Lipschitz function, $A$ and $G$ are mappings specified later,
$\e G(t, u)dW$
is a general diffusion term with $\e\in (0, 1]$
a constant denoting the intensity of the noise.
Originally, such type of equations have a rich
mathematical theory and fundamental applications, see
e.g.  \cite{PR07} for more details.

Random attractor is the effective tool for
capturing the pathwise long-time behavior for random dynamical systems generated by
stochastic differential equations. Based on the theory in \cite{Arnold98}, we know that an It\^{o}-type stochastic ordinary differential
equation generates a random dynamical system under natural assumptions on the
coefficients. The pullback random attractors theory was systematically established
first in \cite{CF94, FS96, Schmalfuss92} and gained its popular applications since the late 1990s,
see e.g. \cite{BLW09, BGLR11, CRC08, CS04, CDF97, CF94, DS03, GLM11, Gess13a, Gess13b, GLWY18, HS09, LGL15, SZS10, Wang11, Wang12, Zhou17}.
The generation of a random dynamical system from an
It\^{o}-type stochastic partial differential equation has been a long-standing
open problem. Most of the efforts are taken to tackle the special cases
when $G(t, u)$ is linear in $u$. When $G(t, u)$ is a general nonlinear
function, we even don't know whether system  \eqref{sys00} generates a random
dynamical system or not. The main reason is that Kolmogorov's theorem fails for random
fields parameterized by
infinite-dimensional Hilbert spaces. This fact leads to dynamical aspects
such as the asymptotic stability,
random attractors, random invariant manifolds have not been fully investigated in generality.

Compared to the pathwise random attractor methods used in
\cite{GGS14, GLS10, GLS16, GGW20, GLW19, GW18, LW17}, one possible theory to deal with
this issue above is to replace it by the mean-square random attractor as suggested in
\cite{KL12}. However, the existence
of mean-square random attractor theory is difficulty in dealing with the general
nonlinear functions in drift term.
Quite recently,
a new type of weak mean-square random attractor, i.e., weak pullback mean random attractor
in reflective Bochner spaces was introduced in \cite{Wang18a, Wang18b} to tackle the stochastic
partial differential equations with general nonlinear drift
and diffusion terms based on the weak topology. Following this line, we first obtain the dynamic
behavior of the abstract system  \eqref{sys00} when the intensity $\e$ small enough
and then apply to some concrete equations such as stochastic reaction-diffusion equations, the
stochastic $p$-Laplace equation and stochastic porous media equations.
It seems that this is the first time to consider establishing these results for
the non-autonomous stochastic evolution equations based on variational approach.
This approach has been used intensively
in recent years to analyze stochastic partial differential equations driven by an infinite-dimensional Wiener process.
For more details, see e.g. the monographs \cite{LR15, PR07} and the references therein.

The paper is arranged as follows. In the next section, we give some basic concepts and results on the existence of weak pullback mean random attractors
for mean random dynamical systems. In Section \ref{sect2}, we first define a mean random dynamical system via the solution operators of the abstract non-autonomous
stochastic evolution equation, and then prove the existence and uniqueness of weak
pullback mean random attractors for this model when the intensity of the noise is small enough. In Section \ref{applications}, some models with different mappings $A$ and appropriate chosen coefficients
to guarantee the  existence of weak
pullback mean random attractors are addressed.

\section{Theory of Mean Random Dynamical Systems}\label{sect1}
We recall some basic concepts and results on the existence of weak pullback mean random attractors
for mean random dynamical systems from \cite{CGS07, KL12, Wang18a, Wang18b}. Let $(\O, \F, \{\F_t\}_{t\in \R}, \P)$
be a complete filtered probability space with $\{\F_t\}_{t\in \R}$ is an increasing right continuous
family of sub-$\s$-algebras of $\F$ that includes all $\P$-null sets. Denote $X$ be a Banach space with
norm $\|\cdot\|_X$. We say a function $\p: \O\r X$ is strongly measurable if there exists a sequence
of simple functions $\p_n: \O\r X$, such that $\lim\limits_{n\r\i}\|\p_n-\p\|_X=0$ $\P$-a.e. Such a function
$\p$ is called Bochner integrable if there exists a sequence of simple functions $\p_n: \O\r X$, such that
$$\lim_{n\r\i}\int_\O\|\p_n-\p\|_Xd\P=0,$$
then the Bochner integral of $\p$ is defined as
$$\int_{\O}\p d\P=\lim_{n\r\i}\int_{\O}\p_n d\P.$$
Now, for $q\ge 1$, we denote $L^q(\O, \F; X)$ the Bochner space consisting of all Bochner integrable functions
$\p: \O\r X$ such that
$$\|\p\|_{L^q(\O, \F; X)}=\Big(\int_\O\|\p\|^q_Xd\P\Big)^{1/q}<\i.$$
For every $\t\in \R$, we define the space $L^q(\O, \F_\t; X)$ analogously.

Let $\D$ be a collection of some families of nonempty bounded subsets of $L^q(\O, \F_\t; X)$
parameterized by $\t\in \R$, that is,
\begin{eqnarray}\label{sub}
\begin{split}
\D=\{D=\{D(\t)\subseteq L^q(\O, \F_\t; X): D(\t)&\neq \emptyset\ \text{bounded},\ \t\in \R\}:\\
&D\ \text{satisfies some conditions}\}.
\end{split}
\end{eqnarray}
Such a collection $\D$ is called inclusion-closed if $D=\{D(\t): \t\in \R\}\in \D$ implies that
every family $\hat{D}=\{\hat{D}(\t): \emptyset\neq \hat{D}(\t)\subseteq D(\t), \forall\t\in \R\}\in \D$.

\begin{definition}\label{def1}
A family $\Phi=\{\phi(t, \t): t\in \R^+, \t\in \R\}$ of mappings is called a mean random dynamical system
on $L^q(\O, \F; X)$ over $(\O, \F, \{\F_t\}_{t\in \R}, \P)$ if for all $\t\in \R$ and $t, s\in \R^+$,
\begin{enumerate}[(i)]
  \item $\phi(t, \t)$ maps $L^q(\O, \F_\t; X)$ to $L^q(\O, \F_{t+\t}; X)$;
  \item $\phi(0, \t)$ is the identity operator on $L^q(\O, \F_\t; X)$;
  \item $\phi(t+s, \t)=\phi(t, \t+s)\circ \phi(s, \t)$.
\end{enumerate}
\end{definition}

\begin{definition}\label{def2}
A family $\K=\{K(\t): \t\in \R\}\in \D$ is called $\D$-pullback absorbing set for $\Phi$
on $L^q(\O, \F; X)$ over $(\O, \F, \{\F_t\}_{t\in \R}, \P)$ if for every $\t\in \R$ and
$D\in \D$, there exists $T=T(\t, D)>0$ such that
$$\phi(t, \t-t)(D(\t-t))\subseteq K(\t)\quad \text{for all}\ t\ge T.$$
Furthermore, for every $\t\in \R$, $K(\t)$ is a weakly compact nonempty subset of $L^q(\O, \F_\t; X)$,
then we call $\K=\{K(\t): \t\in \R\}$ a weakly compact $\D$-pullback
absorbing set for $\Phi$.
\end{definition}

\begin{definition}\label{def3}
A family $\K=\{K(\t): \t\in \R\}\in \D$ is called $\D$-pullback weakly attracting set for $\Phi$
on $L^q(\O, \F; X)$ over $(\O, \F, \{\F_t\}_{t\in \R}, \P)$ if for every $\t\in \R$,
$D\in \D$ and every weak neighborhood $\mathscr{N}^{\rm w}(K(\t))$ of $K(\t)$ in $L^q(\O, \F_\t; X)$,
then there exists $T=T(\t, D, \mathscr{N}^{\rm w}(K(\t)))>0$ such that for all $t\ge T$,
$$\phi(t, \t-t)(D(\t-t))\subseteq \mathscr{N}^{\rm w}(K(\t)).$$
\end{definition}

\begin{definition}\label{def4}
A family $\A=\{\A(\t): \t\in \R\}\in \D$ is called a weak $\D$-pullback mean random attractor
for $\Phi$ on $L^q(\O, \F; X)$ over $(\O, \F, \{\F_t\}_{t\in \R}, \P)$ if for all $\t\in \R$,
\begin{enumerate}[(a)]
  \item $\A(\t)$ is a weakly compact subset of $L^q(\O, \F_\t; X)$;
  \item $\A$ is a $\D$-pullback weakly attracting set of $\Phi$;
  \item $\A$ is the minimal one among $\D$ with both properties (a) and (b).
\end{enumerate}
\end{definition}

The next result was established to guarantee the existence and uniqueness
of weak pullback mean random attractors in $L^q(\O, \F; X)$ over filtered probability space
$(\O, \F, \{\F_t\}_{t\in \R}, \P)$.

\begin{proposition}[See \cite{Wang18a, Wang18b}]\label{prop2}
Assume $X$ is a reflexible Banach space and $q\in (1, \i)$. Let $\D$ be an inclusion-closed
collection of some families of nonempty bounded subsets of $L^q(\O, \F; X)$ as given in
\eqref{sub} and $\Phi$ be a mean random dynamical system on $L^q(\O, \F; X)$ over $(\O, \F, \{\F_t\}_{t\in \R}, \P)$.
If $\Phi$ has a weakly compact $\D$-pullback absorbing set $\K\in \D$ on
$L^q(\O, \F; X)$ over $(\O, \F, \{\F_t\}_{t\in \R}, \P)$, then $\Phi$ has a unique
weak $\D$-pullback mean random attractor $\A\in \D$ with $\A$ given by, for each $\t\in \R$,
$$\A(\t)=\bigcap_{r\ge 0}\overline{\bigcup_{t\ge r}\phi(t, \t-t)(K(\t-t))}^{\rm w},$$
where the closure is taken with respect to the weak topology of $L^q(\O, \F_\t; X)$.
\end{proposition}
\section{Setup and Main Results}\label{sect2}
Let
$$V\subset H\equiv H^*\subset V^*$$
be a Gelfand triple, that is, $H$ is a separable Hilbert space with inner
product $\langle, \rangle_H$ and is identified to its dual $H^*$ due to the Riesz isomorphism,
$V$ is a reflective Banach space such that $V\subset H$ continuously and densely, and denote
$_{V^*}\langle,\rangle_V$ the dualization between $V^*$ and its dual $V$.
In particular, there is a constant $\l>0$ such that
\begin{equation} \label{cont1}
\|v\|_V^2\ge \l \|v\|_H^2\quad \text{for all}\quad v\in V,
\end{equation}
which also implies that when $\a>2$, there exists constant $C_\a>0$ such that
\begin{equation} \label{cont2}
C_\a+\|v\|_V^\a\ge\l\|v\|_H^2.
\end{equation}
Let $(\O, \F, \{\F_t\}_{t\in \R}, \P)$ be a complete filtered probability space
enjoying the usual properties, that is, $\{\F_t\}_{t\in \R}$ is an increasing right continuous
family of sub-$\sigma$-algebras of $\F$ that contains all $\P$-null sets.

Let $\oo$ be a bounded domain in $\R^n$ with smooth boundary $\partial \oo$ and the
Lebesgue measure of the domain $|\oo|\neq 0$.
For $\t\in \R$, consider the following non-autonomous stochastic evolution equations on
$\oo\times (\t, \infty)$ of type
\begin{equation} \label{sys2}
\left\{
\begin{array}{l}\vspace{2mm}
du=\big(A(t, u)+F(u)+g(t, x)\big)dt+\e G(t, u)dW(t),\ x\in \oo,
\\
u(\t)=u_0\in L^2(\O, \F_\t; H),
\end{array}
\right.
\end{equation}
with $W(t)$ a two-sided cylindrical $\mathcal{Q}$-Wiener process with respect to
the filtration $\{\F_t\}_{t\in \R}$, where $\Q$ is the identity operator on another separable Hilbert
space $U$ and with $G$ taking values in $L_2(U, H)$ as denoting the Hilbert space consisting of all Hilbert-Schmidt operators
from $U$ to $H$, but with $A$ taking values in $V^*$. The constant $\e\in (0, 1]$, $g\in L^2_{\rm loc}(\R, H)$ and the stochastic
term is understood in the sense of It\^{o}'s integration.

{\bf (H0)}\ Assume that $F: H\r H$ is a Lipschitz continuous function
with Lipschitz constant $\g_1$ and for all $u\in H$,
\begin{equation} \label{f1}
\|F(u)\|_H\le \g_2(1+\frac 12\|u\|_H),\ \text{with}\ \g_2>0.
\end{equation}

Let
\begin{equation*} \label{sysa1}
A: \R\times V\times \O\r V^*, \quad
G: \R\times V\times \O\r L_2(U, H)
\end{equation*}
are maps such that $A(\cdot, \cdot; \o): \R\times V\r V^*$ and
$G(\cdot, \cdot; \o): \R\times V\r L_2(U, H)$ are respectively, $(\B(\R)\otimes\B(V), \B(V^*))$-
and $(\B(\R)\otimes\B(V), \B(L_2(U, H)))$-measurable for each $\o\in \O$.
We make the following assumptions: for all $v, v_1, v_2\in V$, $\o\in \O$ and $t\in \R$,
\begin{enumerate}[\bf(H1)]
  \item (Hemicontinuity) \ The map
  $$s\mapsto _{V^*}\langle A(t, v_1+s v_2; \o), v\rangle_V$$
  is continuous on $\R$.
  \item (Weak monotonicity)\ There exists $\g_3\in \R$ such that
  \begin{eqnarray*}
\begin{split}
2\g_1\|v_1-v_2\|_H^2&+2_{V^*}\langle A(t, v_1; \o)-A(t, v_2; \o), v_1-v_2\rangle_V\\
&+\|G(t, v_1)-G(t, v_2)\|^2_{L_2(U, H)}\le \g_3\|v_1-v_2\|_H^2.
\end{split}
\end{eqnarray*}
  \item (Coercivity) There exist $\a\ge 2$, $\g_4\in \R$, $\g_5>0$,  and $h_1\in L_{\rm loc}^1(\R)$ such that
  \begin{equation*}
2_{V^*}\langle A(t, v; \o), v\rangle_V+\|G(t, v)\|^2_{L_2(U, H)}\le \g_4\|v\|_H^2-3\g_5\|v\|_V^\a+h_1(t).
\end{equation*}

\item (Growth)\ There exist $\g_6>0$ and $h_2\in L_{\rm loc}^{\frac \a{\a-1}}(\R)$ such that
\begin{equation*}
\|A(t, v; \o)\|_{V^*}\le \g_6\|v\|_V^{\a-1}+h_2(t).
\end{equation*}

\item For the existence of weak pullback attractors, we further assume that
 \begin{equation}\label{h3a1}
\frac{\g_2+|\g_4|}{\g_5}<\l.
\end{equation}

\end{enumerate}

\begin{remark}
By {\bf(H3)} and {\bf(H4)} we know that for all $v\in V$,
\begin{equation}\label{h5}
\begin{aligned}
\|G(t, v)\|&^2_{L_2(U, H)}\le \g_4\|v\|_H^2+(2\g_6-3\g_5)\|v\|_V^{\a}+2h_2(t)\|v\|_V+h_1(t)\\
&\le \g_4\|v\|_H^2+2(\g_6-\g_5)\|v\|_V^{\a}+C_{\g_5, \a}|h_2(t)|^{\frac \a{\a-1}}+h_1(t)\\
&\le \g_4\|v\|_H^2+2\g_6\|v\|_V^{\a}+C_{\g_5, \a}|h_2(t)|^{\frac \a{\a-1}}+h_1(t).
\end{aligned}
\end{equation}
\end{remark}

\begin{definition}\label{solu2}
Let $\t\in \R$. We say that an $H$-valued $\{\F_t\}_{t\in \R}$-adapted
process $\{u(t)\}_{t\in [\t, \i)}$ is a solution to problem
\eqref{sys2} if
$u(\cdot; \o)\in L_{\rm loc}^2([\t, \i), H)\cap L_{\rm loc}^\a((\t, \i), V)$ and $\P$-a.s.
\begin{eqnarray*}
\begin{split}
u(t)
=u_0+&\int_\t^t \big(A(s, u(s); \o)+F(u(s))+g(s)\big)ds\\
&\qquad+\int_\t^t \e G(s, u(s); \o)dW(s), \quad t>\t.
\end{split}
\end{eqnarray*}
\end{definition}

Based on \cite[Theorem 4.2.4]{PR07}, we know that for every $\t\in \R$ and
$u_0\in L^2(\O, \F_\t; H)$, system \eqref{sys2}
has a unique solution $u$ in the sense of Definition \ref{solu2} under the
assumptions {\bf (H0)-(H4)}
and it satisfies, for every $T>0$,
\begin{equation} \label{s2a1}
\E\Big(\int_\t^{\t+T}\|u(s)\|^\a_Vds+\sup_{t\in [\t, \t+T]}\|u(t)\|^2_H\Big)<\i.
\end{equation}
Now, let $\phi$ be a mapping from
$\R^+\times \R\times L^2(\O, H)$ to $L^2(\O, H)$ given by
$$\phi(t, \t) (u_0)=u(t+\t, \t, u_0),$$
where $u$ is the solution of system \eqref{sys2}
with initial data $u_0\in L^2(\O, \F_\t; H)$. Then, we know that
$\Phi=\{\phi(t, \t): t\in \R^+, \t\in \R\}$ is actually a mean random dynamical system on $L^2(\O, \F; H)$
over $(\O, \F, \{\F_t\}_{t\in \R}, \P)$.

We will use $\mathcal{D}$ to denote the collection of all families of
nonempty bounded subsets of $L^2(\O, \F_\t; H)$, i.e.,
\begin{eqnarray*}
\begin{split}
\mathcal{D}=\{D=\{D(\t)\subseteq L^2(\O, \F_\t; H): &D(\t)\neq \emptyset \ \text{bounded},\ \t\in \R\}:\\
&\lim_{\t\r-\i}e^{\l\g_5 \t}\|D(\t)\|^2_{L^2(\O, \F_\t; H)}=0\}.
\end{split}
\end{eqnarray*}
In the sequel, we further assume
\begin{equation}\label{g2}
\int_{-\i}^\t e^{\l\g_5 s}(\|g(s)\|_H^2+|h_1(s)|+|h_2(s)|^{\frac \a{\a-1}})ds<\i, \quad\forall \t\in \R.
\end{equation}

We first derive the following uniform estimates for problem \eqref{sys2}.

\begin{lemma}\label{est1}
Suppose {\bf (H0)-(H5)} and \eqref{g2} hold. Then there exists $\e_0>0$ such that
for every $\e\in(0, \e_0]$ and for every $\t\in \R$ and $D=\{D(t)\}_{t\in \R}\in \mathcal{D}$,
there exists $T=T(\t, D)>0$ such that for all $t\ge T$,
\begin{eqnarray*}
\begin{split}
\E (\|u(&\t, \t-t, u_0)\|^2)\\
&\le L+Le^{-\l\g_5 \t}\int_{-\i}^\t e^{\l\g_5 s}(\|g(s)\|_H^2+|h_1(s)|+|h_2(s)|^{\frac \a{\a-1}})ds,
\end{split}
\end{eqnarray*}
where $u_0\in D(\t-t)$ and $L$ is a positive constant independent of $\t$ and $D$.
\end{lemma}

\begin{proof}
Applying It\^{o}'s formula to \eqref{sys2}, we obtain for $r\ge \t-t$,
\begin{eqnarray*}
\begin{split}
\|u(r, &\ \t-t, u_0)\|_H^2\\
&=\|u_0\|_H^2+2\int_{\t-t}^r {_{V^*}\langle A(s, u(s)), u(s, \t-t, u_0)\rangle_V}ds\\
&\quad+2\int_{\t-t}^r\langle F(u(s)), u(s, \t-t, u_0)\rangle_Hds\\
&\quad+2\int_{\t-t}^r \langle g(s), u(s, \t-t, u_0)\rangle_H ds\\
&\quad+\int_{\t-t}^r \e^2\|G(s, u(s, \t-t, u_0))\|^2_{L_2(U, H)}ds\\
&\quad+2\int_{\t-t}^r\langle u(s, \t-t, u_0), \e G(s, u(s, \t-t, u_0))dW(s)\rangle_H,
\end{split}
\end{eqnarray*}
which implies that
\begin{eqnarray*}
\begin{split}
\E(\|u(r, &\ \t-t, u_0)\|_H^2)\\
&=\E(\|u_0\|_H^2)+2\int_{\t-t}^r \E({_{V^*}\langle A(s, u(s)), u(s, \t-t, u_0)\rangle_V})ds\\
&\quad+2\int_{\t-t}^r\E(\langle F(u(s)), u(s, \t-t, u_0)\rangle_H)ds\\
&\quad+2\int_{\t-t}^r \E(\langle g(s), u(s, \t-t, u_0)\rangle_H) ds\\
&\quad+\int_{\t-t}^r\e^2\E(\|G(s, u(s, \t-t, u_0))\|^2_{L_2(U, H)})ds.
\end{split}
\end{eqnarray*}
Thus, we know that, for almost all $r\ge \t-t$,
\begin{eqnarray}\label{est1a0}
\begin{split}
\frac d{dr}\E(\|&u(r, \t-t, u_0)\|_H^2)=2\E({_{V^*}\langle A(r, u(r)), u(r, \t-t, u_0)\rangle_V})\\
&+2\E(\langle F(u(r)), u(r, \t-t, u_0)\rangle_H)+2\E(\langle g(r), u(r, \t-t, u_0)\rangle_H) \\
&\quad+\e^2\E(\|G(r, u(r, \t-t, u_0))\|^2_{L_2(U, H)}).
\end{split}
\end{eqnarray}
First,  by {\bf (H0)}, we see
\begin{eqnarray*}
\begin{split}
\langle F(u), u\rangle_H&\le \|F(u)\|_H\|u\|_H
\le \g_2\|u\|_H+\frac 12\g_2\|u\|^2_H\\
&\le \frac 1{8}(\l\g_5-\g_2-|\g_4|)\|u\|_H^2+\frac {2\g_2^2}{\l\g_5-\g_2-|\g_4|}+\frac 12\g_2\|u\|_H^2,
\end{split}
\end{eqnarray*}
which implies that
\begin{eqnarray}\label{est1a1}
\begin{split}
2\E(\langle F(u(r)), u(r)\rangle_H)
\le &\frac 14(\l\g_5+3\g_2-|\g_4|)\E(\|u(r, \t-t, u_0)\|_H^2)\\
&+C_{\l,\g_2,\g_4,\g_5}.
\end{split}
\end{eqnarray}
By the Young inequality, we get
\begin{eqnarray}\label{est1a2}
\begin{split}
2&\E(\langle g(r), u(r, \t-t, u_0)\rangle_H) \\
&\le \frac 1{4}(\l\g_5-\g_2-|\g_4|)\E(\|u(r, \t-t, u_0)\|_H^2)+\frac {4}{\l\g_5-\g_2-|\g_4|}\|g(r)\|_H^2.
\end{split}
\end{eqnarray}
Denote
\begin{equation}\label{intensity}
\e_0=\min\left\{1, \sqrt{\frac{\l\g_5-\g_2-|\g_4|}{4\l\g_6+2|\g_4|}}\right\}.
\end{equation}
Then for all $\e\in(0, \e_0]$, by ({\bf H3}) and \eqref{h5} we get
\begin{eqnarray*}
\begin{split}
&2\E({_{V^*}\langle A(r, u(r)), u(r, \t-t, u_0)\rangle_V})+\e^2\E\|G(r, u(r, \t-t, u_0))\|^2_{L_2(U, H)}\\
&\le -2\g_5\E(\|u(r, \t-t, u_0)\|_V^\a)+\g_4\E\|u(r, \t-t, u_0)\|_H^2\\
&\qquad+\e^2\g_4\E\|u(r, \t-t, u_0)\|_H^2+2\e^2\g_6\E(\|u(r, \t-t, u_0)\|_V^\a)\\
&\qquad+C_{\g_5, \a}|h_2(r)|^{\frac \a{\a-1}}
+2|h_1(r)|\\
&\le  -2\g_5\E(\|u(r, \t-t, u_0)\|_V^\a)+\frac{(\l\g_5-\g_2-|\g_4|)\g_6}{2\l\g_6+|\g_4|}\E(\|u(r, \t-t, u_0)\|_V^\a)\\
&\qquad+\left(\frac{(\l\g_5-\g_2-|\g_4|)|\g_4|}{4\l\g_6+2|\g_4|}+|\g_4|\right)\E(\|u(r, \t-t, u_0)\|_H^2)\\
&\qquad+C_{\g_5, \a}|h_2(r)|^{\frac \a{\a-1}}
+2|h_1(r)|,
\end{split}
\end{eqnarray*}
which along with \eqref{est1a0}-\eqref{est1a2} and \eqref{cont2} shows that for almost all $r\ge \t-t$,
\begin{eqnarray}\label{est1a3}
\begin{split}
\frac d{dr}&\E(\|u(r, \t-t, u_0)\|_H^2)\\
&\qquad+\left(2\g_5-\frac{(\l\g_5-\g_2-|\g_4|)\g_6}{2\l\g_6+|\g_4|}\right)\l\E(\|u(r, \t-t, u_0)\|_H^2)\\
&\le \frac 1{2}\left(\l\g_5+\g_2+|\g_4|+\frac{(\l\g_5-\g_2-|\g_4|)|\g_4|}{2\l\g_6+|\g_4|}\right)\E(\|u(r, \t-t, u_0)\|_H^2)\\
&\quad+C_{\l, \g_2, \g_4, \g_5}+C_{\l, \g_2, \g_4, \g_5}\|g(r)\|_H^2+C_{\g_5, \a}|h_2(r)|^{\frac \a{\a-1}}+2|h_1(r)|.
\end{split}
\end{eqnarray}
Obviously, we see
\begin{equation*}
\begin{aligned}
&\left(2\g_5-\frac{(\l\g_5-\g_2-|\g_4|)\g_6}{2\l\g_6+|\g_4|}\right)\l\\
&\qquad\qquad-\frac 1{2}\left(\l\g_5+\g_2+|\g_4|+\frac{(\l\g_5-\g_2-|\g_4|)|\g_4|}{2\l\g_6+|\g_4|}\right)
=\l\g_5.
\end{aligned}
\end{equation*}
Now, by \eqref{est1a3} we obtain, for almost all $r\ge \t-t$,
\begin{eqnarray}\label{est1a4}
\begin{split}
&\frac d{dr}\E(\|u(r, \t-t, u_0)\|_H^2)+\l\g_5\E(\|u(r, \t-t, u_0)\|_H^2)\\
\le &C_{\l, \g_2, \g_4, \g_5}+C_{\l, \g_2, \g_4, \g_5}\|g(r)\|_H^2+C_{\g_5, \a}|h_2(r)|^{\frac \a{\a-1}}+2|h_1(r)|.
\end{split}
\end{eqnarray}
By multiplying \eqref{est1a4} by $e^{\l\g_5 r}$ and integrating over $(\t-t, \t)$ with $t>0$ to show
\begin{eqnarray}\label{est1a5}
\begin{split}
\E(\|u(&\t, \t-t, u_0)\|_H^2)\le e^{-\l\g_5 t}\E(\|u_0\|_H^2)+C_{\l, \g_2, \g_4, \g_5}\\
&+C_{\l, \g_2, \g_4, \g_5,\a}e^{-\l\g_5\t}\int_{\t-t}^\t(\|g(r)\|_H^2+|h_2(r)|^{\frac \a{\a-1}}+|h_1(r)|)dr.
\end{split}
\end{eqnarray}
Since $u_0\in D(\t-t)$ and $D=\{D(t)\}_{t\in \R}\in \D$ we see
\begin{equation*}
\begin{aligned}
e^{-\l\g_5 \t}e^{\l\g_5 (\t-t)}\E(\|u_0\|_H^2)\le e^{-\l\g_5 \t}&e^{\l\g_5 (\t-t)}\|D(\t-t)\|^2_{L^2(\O, \F_\t; H)}
\rightarrow 0\\
& \ \text{as}\ t\rightarrow\infty,
\end{aligned}
\end{equation*}
which implies that there exists $T_1=T_1(\t, D)>0$ such that for all $t\ge T_1$,
\begin{eqnarray}\label{est1a6}
e^{-\l\g_5 \t}e^{\l\g_5 (\t-t)}\E(\|u_0\|_H^2)\le 1.
\end{eqnarray}
By \eqref{est1a5}-\eqref{est1a6} we have, for all $t\ge T_1$,
\begin{eqnarray*}
\begin{split}
\E(\|u(&\t, \t-t, u_0)\|_H^2)\le 1+C_{\l, \g_2, \g_4, \g_5}\\
&+C_{\l, \g_2, \g_4, \g_5,\a}e^{-\l\g_5\t}\int_{-\infty}^\t(\|g(r)\|_H^2+|h_2(r)|^{\frac \a{\a-1}}+|h_1(r)|)dr.
\end{split}
\end{eqnarray*}
The proof is complete.
\end{proof}

Based on Lemma \ref{est1}, we establish the existence of weakly compact $\D$-pullback absorbing
set for \eqref{sys2}.

\begin{lemma}\label{absorbing set}
Suppose {\bf (H0)-(H5)} and \eqref{g2} hold. Then there exists $\e_0>0$ such that
for every $\e\in(0, \e_0]$, the random dynamical system $\Phi$  for problem \eqref{sys2}
possesses a weakly compact $\D$-pullback absorbing set $\K=\{K(\t): \t \in \R \}\in \D$,
which is given by, for each $\t\in\R$,
\begin{eqnarray}\label{est2a1}
K(\t)=\{u\in L^2(\O, \F_\t; H): \E(\|u\|^2_H)\le R(\t)\},
\end{eqnarray}
where
\begin{eqnarray*}
R(\t)= L+Le^{-\l\g_5 \t}\int_{-\i}^\t e^{\l\g_5 s}(\|g(r)\|_H^2+|h_1(s)|+|h_2(s)|^{\frac \a{\a-1}})ds,
\end{eqnarray*}
with $L$ being the same constant as in Lemma \ref{est1}.
\end{lemma}

\begin{proof}
Since for each $\t\in \R$, $K(\t)$ defined in \eqref{est2a1} is a bounded closed convex subset of $L^2(\O, \F_\t; H)$,
 and hence it is weakly compact in $L^2(\O, \F_\t; H)$. Moreover, by Lemma \ref{est1}, we know
 that for every $\t\in \R$ and $D=\{D(t)\}_{t\in \R}\in \D$, there exists $T=T(\t, D)>0$ such that for
 all $t\ge T$ and $0<\e \le \e_0$,
$$\Phi(t, \t-t, D(\t-t))\subseteq K(\t). $$
Obviously by \eqref{g2}, we can verify $\K\in \D$. Therefore, $\K$ is a weakly compact
$\D$-pullback absorbing set for $\Phi$.
\end{proof}

Now, by Lemma \ref{absorbing set} and Proposition \ref{prop2}, we obtain the main
result of the existence of weak $\D$-pullback mean random attractor for \eqref{sys2}.

\begin{theorem}\label{main1}
Suppose {\bf (H0)-(H5)} and \eqref{g2} hold. Then there exists $\e_0>0$ such that
for every $\e\in(0, \e_0]$, the random dynamical system $\Phi$  for problem \eqref{sys2}
possesses a unique weak $\D$-pullback mean random attractor $\A=\{\A(\t): \t\in \R\}\in \D$
in $L^2(\O, \F_\t; H)$ over $(\O, \F, \{\F_\t\}_{t\in \R}, \P)$.
\end{theorem}

\section{Some Applications}\label{applications}

In this section, we present some  examples of concrete stochastic evolution equations.
Here, we solely focus on $A$ independent of $t$ and take $G\equiv 0$. The latter we do
because of the fact in \cite[Exercise 4.1.2]{PR07} . We also note that
from here examples for $A$ dependent on $(t, \o)$ are then immediate.
We always consider $\oo\subset \R^n$ as an open bounded domain
with smooth boundary $\partial \oo$ and the external forcing term $g\in L^2_{\rm loc}(\R, H)$.

\begin{example}[$H_0^1\subset L^2\subset (H_0^1)^*$ and $A=\Delta$]
Consider the following triple
\begin{equation*}
\begin{aligned}
V:=H_0^1(\oo)\subset L^2(\oo):=H\subset (H_0^1(\oo))^*=H^{-1}(\oo):=V^*
\end{aligned}
\end{equation*}
and the stochastic reaction-diffusion equations
\begin{equation} \label{sys-e0}
\left\{
\begin{array}{l}\vspace{2mm}
du=(\Delta u+F(u)+g(x, t))dt+\e G(u)dW,\quad x\in \oo, \ t>\t,\\
u(x, t)=0, \quad x\in \partial \oo,\ t>\t,\\
u(x, \t)=u_0\in L^2(\O, \F_\t; L^2(\oo)).
\end{array}
\right.
\end{equation}
Let $A(u)=\Delta u$. Similarly to \cite[Example 4.1.7]{PR07},  we can extend $A$
with initial domain $C_0^\infty(\oo)$ to a bounded linear operator $A: V\rightarrow V^*$.
Since $A: V\rightarrow V^*$ is linear, {\bf (H1)} is obviously satisfied.
In this case, we don't need \eqref{h5} due to
$$2_{V^*}\langle A(u), u\rangle_V=-\|\nabla u\|^2_{L^2(\oo)}\le 0.$$
So the conditions given in \eqref{h3a1} and \eqref{intensity} should be adjusted as
\begin{equation}\label{intensity2}
\frac {\g_2}{\g_5}<\l, \quad \tilde\e_0=\min\left\{1, \sqrt{\frac{\l\g_5-\g_2}{4\l\g_6+2|\g_4|}}\right\}.
\end{equation}
Indeed, the $\l$ defined in \eqref{cont1} can be chosen as
\begin{equation*}
\|v\|^2_{L^2(\oo)}\le \sqrt[n]{\frac{|\oo|}{\mu_n}} \|\nabla v\|^2_{L^2(\oo)}\quad \text{for all}\quad v\in H_0^1(\oo),
\end{equation*}
where $\mu_n$ is the volume of unit ball in $\R^n$ in terms of the Gamma function $\Gamma$:
$$   \mu_n=\frac {2\pi^{n/2}}{n\Gamma(\frac n 2)}.$$

Now, take $\g_1\ge 0$, $0<\g_2<\frac 23\l$, $\g_3=2\g_1$, $\g_4=2$, $\g_5=\frac 23$, $\g_6=1$
and $\a=2$, then all the conditions in {\bf (H2)}-{\bf (H4)} and \eqref{intensity2} are satisfied.
Then, we have
\begin{theorem}\label{main2}
Assume that $\g_1\ge 0$, $0<\g_2<\frac 23\l$, $\g_3=2\g_1$, $\g_4=2$, $\g_5=\frac 23$, $\g_6=1$
and $\a=2$ in
{\bf (H0)-(H5)} and \eqref{g2} hold. Then there exists $\tilde\e_0=\min\left\{1, \sqrt{\frac{\frac 23\l-\g_2}{4\l+4}}\right\}$ such that
for every $\e\in(0, \tilde\e_0]$, the random dynamical system $\Phi$  for problem \eqref{sys-e0}
possesses a unique weak $\D$-pullback mean random attractor $\A=\{\A(\t): \t\in \R\}\in \D$
in $L^2(\O, \F_\t; L^2(\oo))$ over $(\O, \F, \{\F_\t\}_{t\in \R}, \P)$.
\end{theorem}

We remark that the similar result has been established in \cite{Wang18a}.
\end{example}

\begin{example}[$L^p\subset L^2\subset L^{p/p-1}$ and $A(u):=-u|u|^{p-2}$]
For $p\ge 2$, consider the following triple
\begin{equation*}
\begin{aligned}
V:=L^p(\oo)\subseteq L^2(\oo):=H\subseteq (L^p(\oo))^*=L^{\frac p{p-1}}(\oo):=V^*
\end{aligned}
\end{equation*}
and the stochastic evolution equation
\begin{equation} \label{sys-e1}
\left\{
\begin{array}{l}\vspace{2mm}
du=(-u|u|^{p-2}+F(u)+g(t, x))dt+\e G(u)dW,\quad x\in \oo, \ t>\t,\\
u(x, t)=0, \quad x\in \partial \oo,\ t>\t,\\
u(x, \t)=u_0\in L^2(\O, \F_\t; L^2(\oo)).
\end{array}
\right.
\end{equation}
Define $A: V\rightarrow V^*$ by
$$A(u):=-u|u|^{p-2}, \quad u\in V.$$
Indeed, $A$ takes values in $V^*$ for all $u\in V$. Furthermore, $A$ satisfies {\bf (H1)-(H4)}.
This can be verified by the same steps as in \cite[Example 4.1.5]{PR07} with
\begin{itemize}
\item $\g_3:=2\g_1$ in {\bf (H2)};
\item $\a:=p$, $\g_4=0$ and $\g_5=\frac 23$ in {\bf (H3)};
\item $\a:=p$ and $\g_6=1$ in {\bf (H4)}.
\end{itemize}
From \eqref{cont1}, we know that for $p>2$,
\begin{equation*}
\|v\|_{L^2(\oo)}^2\le \|v\|_{L^p(\oo)}^2|\oo|^{\frac p{p-2}} \quad \text{for all}\quad v\in L^p(\oo),
\end{equation*}
which implies that we can choose $\g_2>0$ given in \eqref{f1} such that
 \begin{equation*}
\frac{\g_2+|\g_4|}{\g_5}=\frac 32\g_2<\l_0
\Rightarrow \g_2<\frac {2}{3}\l_0,
\end{equation*}
where $\l_0$ depends on $|\oo|$ and $p$.

Then, we have the existence of  weak $\D$-pullback mean random attractor for \eqref{sys-e1}.
\begin{theorem}\label{main2}
Assume that $\g_1\ge 0$, $0<\g_2<\frac {2}{3}\l_0$,
$\g_3=2\g_1$, $\g_4=0$, $\g_5=\frac 23$, $\g_6=1$, $\a=p$ in
{\bf (H0)-(H5)} and \eqref{g2} hold. Then there exists $\tilde\e_0=\min\left\{1, \sqrt{\frac{\frac 23\l_0-\g_2}{4\l_0}}\right\}$ such that
for every $\e\in(0, \tilde\e_0]$, the random dynamical system $\Phi$  for problem \eqref{sys-e1}
possesses a unique weak $\D$-pullback mean random attractor $\A=\{\A(\t): \t\in \R\}\in \D$
in $L^2(\O, \F_\t; L^2(\oo))$ over $(\O, \F, \{\F_\t\}_{t\in \R}, \P)$.
\end{theorem}

\end{example}

\begin{example}[$W_0^{1,p}\subset L^2\subset (W_0^{1,p})^*$ and $A=p$-Laplacian]
Again, we take $p\ge 2$, consider the following triple
\begin{equation*}
\begin{aligned}
V:=W_0^{1,p}(\oo)\subseteq L^2(\oo):=H\subseteq (W_0^{1,p}(\oo))^*:=V^*
\end{aligned}
\end{equation*}
and the stochastic $p$-Laplacian equation
\begin{equation} \label{sys-e2}
\left\{
\begin{array}{l}\vspace{2mm}
du={\rm div}(|\nabla u|^{p-2}\nabla u)+F(u)+g(t, x))dt+\e G(u)dW,\quad x\in \oo, \ t>\t,\\
u(x, t)=0, \quad x\in \partial \oo,\ t>\t,\\
u(x, \t)=u_0\in L^2(\O, \F_\t; L^2(\oo)).
\end{array}
\right.
\end{equation}
Define $A: V\rightarrow V^*$ by
$$A(u):={\rm div}(|\nabla u|^{p-2}\nabla u), u\in V.$$
Indeed, $A$ takes values in $V^*$ for all $u\in V$.
Similarly to \cite[Example 4.1.9]{PR07}, {\bf (H1)}  is obvious.
We also have
\begin{equation*}
\tilde\l\|v\|_{L^p(\oo)}\le \|\nabla v\|_{L^p(\oo)}\quad \text{for all}\quad v\in W_0^{1, p}(\oo),
\end{equation*}
where $ \tilde\l=\sqrt[n]{\frac{\mu_n}{|\oo|}}$.
Now, by taking
\begin{itemize}
\item $\g_1\ge 0$, $0<\g_2<\frac{\tilde \l}6\min\{1, \tilde \l\}$ in {\bf (H0)};
\item $\g_3:=2\g_1$ in {\bf (H2)};
\item $\a:=p$, $\g_4=0$ and $\g_5=\frac{1}6\min\{1, \tilde \l\}$ in {\bf (H3)};
\item $\a:=p$ and $\g_6=1$ in {\bf (H4)},
\end{itemize}
we establish the following result:
\begin{theorem}\label{main3}
Assume that $\g_1\ge 0$, $0<\g_2<\frac{\tilde \l}6\min\{1, \tilde \l\}$,
$\g_3=2\g_1$, $\g_4=0$, $\g_5=\frac{1}6\min\{1, \tilde \l\}$, $\g_6=1$, $\a=p$ in
{\bf (H0)-(H5)} and \eqref{g2} hold. Then there exists $\tilde\e_0=\min\left\{1, \sqrt{\frac{\frac{\tilde \l}6\min\{1, \tilde \l\}-\g_2}{4\tilde\l}}\right\}$ such that
for every $\e\in(0, \tilde\e_0]$, the random dynamical system $\Phi$  for problem \eqref{sys-e2}
possesses a unique weak $\D$-pullback mean random attractor $\A=\{\A(\t): \t\in \R\}\in \D$
in $L^2(\O, \F_\t; L^2(\oo))$ over $(\O, \F, \{\F_\t\}_{t\in \R}, \P)$.
\end{theorem}

\end{example}

\begin{example}[$L^p\subset (H_0^1)^*\subset (L^p)^*$ and $A=$ porous medium operator]
For $p\ge 2$, consider the following triple
\begin{equation*}
\begin{aligned}
V:=L^{p}(\oo)\subseteq (H_0^1(\oo))^*:=H\subseteq (L^p(\oo))^*=L^{\frac p{p-1}}(\oo):=V^*
\end{aligned}
\end{equation*}
and the stochastic porous medium equations
\begin{equation} \label{sys-e3}
\left\{
\begin{array}{l}\vspace{2mm}
du=\Delta \Psi(u)+F(u)+g(t, x))dt+\e G(u)dW,\quad x\in \oo, \ t>\t,\\
u(x, t)=0, \quad x\in \partial \oo,\ t>\t,\\
u(x, \t)=u_0\in L^2(\O, \F_\t; H^{-1}(\oo)).
\end{array}
\right.
\end{equation}
Here, the function $\Psi: \R\rightarrow \R$ possesses the following properties:
\begin{enumerate}[\bf($\Psi 1$)]
\item $\Psi$ is continuous.
\item For all $s, t\in \R$,
$$(t-s)(\Psi(t)-\Psi(s))\ge 0.$$
\item There exist $p\ge 2$, $\b_1>0$, $\b_2\ge 0$ such that for all $s\in \R$,
$$s\Psi(s)\ge \b_1|s|^p-\b_2.$$
\item There exist $\b_3,\b_4>0$ such that for all $s\in \R$,
$$|\Psi(s)|\le \b_3|s|^{p-1}+\b_4.$$
\end{enumerate}
Define the porous medium operator $A: L^{p}(\oo)\rightarrow (L^p(\oo))^*$ by
$$A(u):=\Delta \Psi(u), \quad u\in L^{p}(\oo).$$
Then the operator is well-defined, see e.g. \cite[Example 4.1.11]{PR07}.
Also, {\bf (H1)} holds. By {\bf ($\Psi 2$)}, $\g_3:=2\g_1$ in {\bf (H2)}.
{\bf (H3)} is satisfied with $\g_4=0$, $\g_5=\frac 23\b_1$, $\a=p$ and $h_1(t):=2\b_2|\oo|$.
Take $\a=p$, $\g_6=\b_3$ and $h_2(t):=\b_4|\oo|^{p/p-1}$, then {\bf (H4)} is satisfied.
Furthermore, $\g_2$ can be chosen such that
\begin{equation*}
\g_2<\frac 23\b_1\hat\l,
\end{equation*}
where $\hat\l$ defined in
\begin{equation*}
\|v\|_{L^p(\oo)}^2\ge \hat\l \|v\|_{H^{-1}(\oo)}^2\quad \text{for all}\quad v\in L^p(\oo).
\end{equation*}
Then we have the main result:
\begin{theorem}\label{main4}
Assume that $\g_1\ge 0$, $0<\g_2<\frac 23\b_1\hat\l$,
$\g_3=2\g_1$, $\g_4=0$, $\g_5=\frac 23\b_1$, $\g_6=\b_3$, $\a=p$, $h_1(t):=2\b_2|\oo|$,  $h_2(t):=\b_4|\oo|^{p/p-1}$ in
{\bf (H0)-(H5)} and \eqref{g2} hold. Then there exists $\tilde\e_0=\min\left\{1, \sqrt{\frac{\frac 23\b_1\hat\l-\g_2}{4\hat\l\b_3}}\right\}$ such that
for every $\e\in(0, \tilde\e_0]$, the random dynamical system $\Phi$  for problem \eqref{sys-e3}
possesses a unique weak $\D$-pullback mean random attractor $\A=\{\A(\t): \t\in \R\}\in \D$
in $L^2(\O, \F_\t; H^{-1}(\oo))$ over $(\O, \F, \{\F_\t\}_{t\in \R}, \P)$.
\end{theorem}

\end{example}

\bibliographystyle{amsplain}

\begin{thebibliography}{40}

\bibitem{Arnold98}
L. Arnold, \textit{Random Dynamical Systems}, Springer-Verlag, Berlin 1998.


\bibitem{BLW09}
P.W. Bates, K. Lu and B. Wang,
\textit{Random attractors for stochastic reaction-diffusion equations on unbounded
domains}, J.  Differential Equations {\bf 246} (2009) 845--869.

\bibitem{BGLR11}
W.J. Beyn, B. Gess, P. Lescot and M. R\"{o}ckner,
\textit{The global random attractor for a class of stochastic porous
media equations}, { Commun. Partial Differ. Equ.} {\bf 36}  (2011) 446--469.



\bibitem{CGS07}
T. Caraballo, M.J. Garrido-Atienza and B. Schmalfuss, \textit{Existence of exponetially
attracting stationary solutions for delay evolution equations},
{ Discrete Contin. Dyn. Syst.} {\bf 21} (2007) 271--293.




\bibitem{CRC08}
T. Caraballo, J. Real and I. Chueshov,
\textit{Pullback attractors for stochastic heat
equations in materials with memory},
{Discrete Contin. Dyn. Syst. Ser. B}
{\bf 9}  (2008)  525--539.


\bibitem{CS04}
I. Chueshov and M. Scheutzow,
\textit{On the structure of attractors  and invariant measures for a class of
monotone  random systems},
{ Dynamical Systems}  {\bf 19} (2004)  127--144.



\bibitem{CDF97} H. Crauel,  A. Debussche and
F. Flandoli, \textit{Random attractors},
{ J. Dyn. Diff. Equat. }  {\bf 9} (1997)  307--341.


\bibitem{CF94}
H. Crauel and F. Flandoli,
\textit{Attractors for random dynamical systems},
{ Probab. Th. Re. Fields}  {\bf 100} (1994)  365--393.


\bibitem{DS03}
J. Duan and B. Schmalfuss,
\textit{The 3D quasigeostrophic fluid dynamics
  under random forcing on boundary},
  {Comm. Math. Sci.}
  {\bf 1} (2003)   133--151.


\bibitem{FG19}
B. Fehrman and B. Gess,
\textit{Well-posedness of nonlinear diffusion
equations with nonlinear, conservative noise},
{ Arch. Rational Mech. Anal.} {\bf233} (2019) 249--322.

\bibitem{FS96}
F. Flandoli and B. Schmalfuss,
\textit{Random attractors for
the 3D stochastic Navier-Stokes equation with multiplicative
noise},
    { Stoch. Stoch. Rep.}  {\bf 59} (1996)   21--45.

\bibitem{GGS14}
H. Gao, M.J. Garrido-Atienza and B. Schmalfuss,
\textit{Random attractors for
stochastic evolution equations driven by fractional Brownian motion}, { SIAM
J. Math. Anal.} {\bf 46} (2014) 2281--2309.



\bibitem{GLS10}
M.J. Garrido-Atienza, K. Lu and B. Schmalfuss,
\textit{Random dynamical systems
for stochastic partial differential equations driven by a fractional
Brownian motion}, {Discrete Contin. Dyn. Syst.} {\bf14} (2010) 473--493.



\bibitem{GLS16}
M.J. Garrido-Atienza, K. Lu and  B. Schmalfuss,
\textit{Random dynamical systems for stochastic
evolution equations driven by multiplicative fractional Brownian noise with Hurst parameters
$H\in (1/3, 1/2]$}, {SIAM J. Appl. Dyn. Syst.} {\bf15} (2016) 625--654.



\bibitem{GLM11}
B. Gess,  W. Liu and
M. R\"{o}ckner,
\textit{Random attractors  for  a class
of stochastic partial differential equations
driven by general additive noise},
{J.  Differential Equations}
{\bf  251} (2011)   1225--1253.



\bibitem{Gess13a}
B. Gess,
\textit{Random attractors  for
degenerate stochastic partial differential equations},
{ J. Dyn. Diff. Equat. }
{\bf  25} (2013)   121--157.


\bibitem{Gess13b}
B. Gess,
\textit{Random attractors  for singular stochastic evolution equations},
{J. Differential Equations}  {\bf 255} (2013)
524--559.



\bibitem{GGW20}
A. Gu, B. Guo and B. Wang,
\textit{Long term behavior of random Navier-Stokes equations driven by colored noise},
{ Discrete Contin. Dyn. Syst. Ser. B} {\bf 25} (2020) 2495--2532.





\bibitem{GLW19}
A. Gu, K. Lu and B. Wang, \textit{Asymptotic behavior of random Navier-Stokes
equations driven by Wong-Zakai approximations}, { Discrete Contin. Dyn.
Syst. Ser. A} {\bf 39} (2019) 185--218.




\bibitem{GLWY18}
A. Gu, D. Li, B. Wang and H. Yang,
\textit{Regularity of random attractors for fractional stochastic reaction-diffusion equations on $\R^n$},
{J. Differential Equations}  {\bf 264} (2018)
7094--7137.

\bibitem{GW18}
A. Gu and B. Wang, \textit{Asymptotic behavior of random FitzHugh-Nagumo
systems driven by colored noise}, {Discrete Contin. Dyn. Syst. Ser. B} {\bf23}
(2018) 1689--1720.




\bibitem{HS09}
J. Huang and W. Shen,
\textit{Pullback attractors for nonautonomous and random
parabolic equations on non-smooth domains},
{Discrete Contin. Dyn. Syst.}
{\bf 24} (2009)    855--882.



\bibitem{KL12}
P.E. Kloeden and T. Lorenz, \textit{Mean-square random dynamical systems},
{ J. Differential Equations} {\bf 253} (2012) 1422--1438.




\bibitem{LGL15}
Y. Li, A. Gu and J. Li,
\textit{Existence and continuity of bi-spatial random attractors and application to stochastic
semilinear Laplacian equations},
{J. Differential Equations} {\bf 258} (2015) 504--534.


\bibitem{LR15}
W. Liu and M. R\"{o}ckner,
\textit{Stochastic Partial Differential Equations: An Introduction},
Universitext, Springer, Switzerland, 2015.


\bibitem{LW17}
K. Lu and B. Wang, \textit{Wong-Zakai approximations and long term behavior
of stochastic partial differential equations}, { J. Dyn. Diff. Equat.}
{\bf 31} (2019) 1341--1371.







\bibitem{PR07}
C. Pr\'{e}v\^{o}t and M. R\"{o}ckner,
\textit{A Concise Course on Stochastic Partial Differential Equations}, Lecture Notes
in Mathematics, Springer, Berlin, 2007.


\bibitem{Schmalfuss92}
B.  Schmalfuss, \textit{Backward cocycles  and attractors
of stochastic differential equations},
{\em International Seminar on Applied Mathematics-Nonlinear
Dynamics: Attractor Approximation and Global Behavior},    185-192,
Dresden,
1992.

\bibitem{SZS10}
Z. Shen, S. Zhou and W. Shen,
\textit{One-dimensional random attractor and rotation number of the stochastic
damped sine-Gordon equation}, {J. Differential Equations} {\bf 248} (2010) 1432--1457.




 \bibitem{Wang11}
 B.  Wang,
\textit{Asymptotic behavior of stochastic wave equations with critical
exponents on $\R^3$},
{ Trans. Amer. Math. Soc.}
{\bf 363} (2011)  3639--3663.




\bibitem{Wang12}
      B. Wang,
      \textit{Sufficient and necessary criteria for
      existence of pullback attractors for
      non-compact random dynamical systems},
     { J. Differential Equations}
      {\bf 253}  (2012) 1544--1583.





\bibitem{Wang18a}
B. Wang,
\textit{Weak pullback attractors for mean random dynamical systems in Bochner spaces},
{ J. Dyn. Diff. Equat.} {\bf 31} (2019) 2177--2204.

\bibitem{Wang18b}
B. Wang, \textit{Weak pullback attractors for stochastic Navier-Stokes
equations with nonlinear diffusion terms}, { Proc. Amer. Math. Soc.}
{\bf 147} (2019) 1627--1638.



\bibitem{Zhou17}
S. Zhou, \textit{Random exponential attractor for stochastic
reaction-diffusion equation with multiplicative
noise in $\R^3$}, { J. Differential Equations} {\bf 263} (2017) 6347--6383.

\end{thebibliography}

\end{document}